\documentclass[10pt,reqno]{amsart}
\usepackage{color}
\usepackage{enumerate}

\newtheorem{theorem}{Theorem}[section]
\newtheorem{lemma}[theorem]{Lemma}
\newtheorem{corollary}[theorem]{Corollary}

\theoremstyle{remark}
\newtheorem{remark}[theorem]{Remark}

\theoremstyle{definition}
\newtheorem{assumption}[theorem]{Assumption}
\newtheorem{example}[theorem]{Example}
\newtheorem{definition}[theorem]{Definition}

\newcommand\cbrk{\text{$]$\kern-.15em$]$}}
\newcommand\opar{\text{\,\raise.2ex\hbox{${\scriptstyle|}$}\kern-.34em$($}}
\newcommand\cpar{\text{$)$\kern-.34em\raise.2ex\hbox{${\scriptstyle |}$}}\,}

\def\qed{{\hfill $\Box$ \bigskip}}

\def\XXint#1#2#3{{\setbox0=\hbox{$#1{#2#3}{\int}$}
\vcenter{\hbox{$#2#3$}}\kern-.5\wd0}}

\newcommand\bL{\mathbb{L}}
\newcommand\bH{\mathbb{H}}
\newcommand\bZ{\mathbb{Z}}

\newcommand\bE{\mathbb{E}}

\newcommand\bN{\mathbb{N}}

\newcommand\fR{\mathbf{R}}

\newcommand\cF{\mathcal{F}}

\newcommand\cI{\mathcal{I}}

\newcommand\cS{\mathcal{S}}
\newcommand\cM{\mathcal{M}}
\newcommand\cT{\mathcal{T}}
\newcommand\cO{\mathcal{O}}

\newcommand{\mysection}[1]{\section{#1}
\setcounter{equation}{0}}

\begin{document}

\title[Degenerate PDE$\mathbf{s}$]
{On  the second order derivative estimates for  degenerate parabolic equations}

\author{Ildoo Kim}
\address{Department of mathematics, Korea university, 1 anam-dong
sungbuk-gu, Seoul, south Korea 136-701}
\email{waldoo@korea.ac.kr}
\thanks{The research of the first author was supported by Basic
Science Research Program through the National Research Foundation of
Korea(NRF) grant funded by the Korea
government(2017R1C1B1002830)}

\author{Kyeong-Hun Kim}
\address{Department of mathematics, Korea university, 1 anam-dong
sungbuk-gu, Seoul, south Korea 136-701}
\email{kyeonghun@korea.ac.kr}

\thanks{The research of the second author was supported by Basic
Science Research Program through the National Research Foundation of
Korea(NRF) grant funded by the Korea
government (2017R1D1A1B03033255)}

\subjclass[2010]{35K65, 35B65, 35K15}

\keywords{Time degenerate parabolic equations, Maximal $L_p$-regularity,  initial-value problem}

\begin{abstract}
We study the  parabolic  equation 
\begin{align}
				\notag
&u_t(t,x)=a^{ij}(t)u_{x^ix^j}(t,x)+f(t,x), \quad (t,x) \in [0,T] \times \fR^d  \\
&u(0,x)=u_0(x) 
				\label{main eqn}
\end{align}
 with the  full degeneracy of the leading coefficients, that is,
 \begin{align}
					\label{ab el}
(a^{ij}(t)) \geq \delta(t)I_{d\times d} \geq 0.
\end{align}
It is well known that if $f$ and $u_0$ are not smooth enough, say $f\in \bL_p(T):=L_p([0,T] ; L_p(\fR^d))$ and $u_0\in L_p(\fR^d)$, then in general the solution is only in $C([0,T];L_p(\fR^d))$, and thus derivative estimates  are not possible.

In this article we prove that  $u_{xx}(t,\cdot)\in L_p(\fR^d)$ on the set $\{t: \delta(t)>0 \}$ and 
\begin{align*}
\int^T_0 \|u_{xx}(t)\|^p_{L_p} \delta(t)dt\leq 
N(d,p) \left(\int^T_0 \|f(t)\|^p_{L_p}\delta^{1-p}(t)dt + \|u_0\|^p_{B^{2-2/ p}_p} \right),
\end{align*}
where $B^{2-2/ p}_p$ is the Besov space of order $2-2/p$.
We also prove that $u_{xx}(t,\cdot)\in L_p(\fR^d)$ for all $t>0$ and 
\begin{equation}
    \label{10.13.3}
\int^T_0 \|u_{xx}\|^p_{L_p(\fR^d)}\,dt \leq N \|u_0\|^p_{B^{2-2/(\beta p)}_p},
\end{equation}
if $f=0$, $\int^t_0 \delta(s)ds>0$ for each $t>0$,  and  a certain asymptotic behavior of $\delta(t)$ holds near $t=0$ (see (\ref{eqn extra})). Here 
$\beta>0$ is the constant related to the asymptotic behavior in (\ref{eqn extra}).  For instance, if $d=1$ and $a^{11}(t)=\delta(t)=1+\sin(1/t)$, then 
(\ref{10.13.3}) holds with $\beta=1$, which actually equals the maximal regularity of the heat equation $u_t=\Delta u$.
\end{abstract}

\maketitle

\mysection{introduction}

The study of  degenerate  elliptic and parabolic equations was started long time ago.  For instance, $L_2$-theory for the fully degenerate elliptic and parabolic equation was developed in \cite{O,O1,O2,O3}, and $L_p$-theory was introduced in \cite{gerencser2015solvability, krylov1986characteristics}.  In these articles, it is assumed that  $A:=(a^{ij})\geq 0$ and it depends both $t$ and $x$.  Under the such general degeneracy, the solution is only in $C([0,T]; L_p)$ and  can not be smoother than $f$ and $u_0$. This can be easily seen by taking $A=0$ so that 
\begin{equation}
      \label{extreme}
u(t)=u(0)+\int^t_0 fds, \quad \forall \, t>0.
\end{equation}
Thus, some extra conditions on the degeneracy are needed for  better regularity  of solutions.   Especially there are many articles  handling the degeneracy depending mainly on the space variable. See e.g.\cite{oleinik2012second,kohn1967degenerate,kim2007sobolev,fornaro2012degenerate} (analtyic methods) and \cite{stroock1972degenerate,freidlin1969smoothness} (probabilistic method).
These results focused on controlling the degeneracy of $a^{ij}(x)$ near the boundary of domains and used certain weights or considered splitting the boundary with a modified Dirichlet condition.

  In this article we investigate a maximal regularity of solutions under the condition that the degeneracy depends only on the time variable,
   that is $a^{ij}=a^{ij}(t)$. 
    Based on a standard perturbation argument one can extend our result to  the general case
    $a^{ij}=a^{ij}(t,x)$ (see Remark \ref{variable x}).  
    As mentioned in the abstract above, under  condition (\ref{ab el}),  we prove that  $u_{xx}(t)\in L_p$ on $\{t: A(t)>0\}$ and
 \begin{equation}
      \label{eqn second}
\|u_{xx}\|_{L_p([0,T], \delta(t) dt; L_p)} \leq N(d,p)  \left(\|f\|_{L_p([0,T],\delta^{1-p}(t)dt ; L_p)}+\|u_0\|_{B^{2-2/ p}_p} \right).
\end{equation}
  This might look absurd at  first glance if one considers the above example (see (\ref{extreme})), but in such case (\ref{eqn second}) is trivial because the left hand side is zero.

To estimate $u_{xx}$ without the help of the weight $\delta(t)dt$,  we  further assume  that   $\int^t_0 \delta(s) ds>0$ for any $t>0$,  $|a^{ij}(t)|\leq N \delta(t)$, and for some $\beta, t_0, N_0 >0$,
\begin{align}
					\label{eqn extra}
\left| \left\{ t \in [0,t_0] : h \leq \int_0^t \delta (s) ~ds <4h \right\} \right| \leq N_0 h^{1/\beta}, \quad \forall h \approx 0.
\end{align}
Under these conditions we prove
\begin{equation}
        \label{eqn zerof}
        \int^T_0 \|u_{xx}\|^p_{L_p}dt \leq N \|u_0\|^p_{B^{2\left(1-\frac{1}{\beta p} \right)}_p}, \quad \text{provided}\, f=0.
        \end{equation}
 If, for instance,
          \begin{equation}
                              \label{eqn ex1}     
                   t^{\beta_0} \leq  c\int^t_0 \delta(s)ds,  \quad \forall \, t>0
            \end{equation}          for some $c, \beta_0>0$, then 
       $ \left\{ t \in [0,1] : h \leq \int_0^t \delta (s) ~ds <4h \right\} \subset  \left[0, (4h)^{1/{\beta_0}} \right] $,
     and thus (\ref{eqn extra}) holds with $\beta=\beta_0$.  One can check that e.g. $\delta(t)=1+\sin(1/t)$ satisfies (\ref{eqn ex1}) 
     with $c=2$ and $\beta_0=1$ (see  Examples \ref{ex 3} and  \ref{ex 4} for detail and more examples).

Actually, (\ref{eqn second}) was already introduced in  \cite{kim2017heat} if $\delta(t)=t^{\alpha}$, where $\alpha>-1$.  Note that if $\delta(t)=t^{\alpha}$ then $\int^t_0 \delta(s)ds=(1+\alpha)^{-1}t^{\alpha+1}$ and thus (\ref{eqn ex1}) holds with $\beta_0=\alpha+1$.   We remark that our approach is somewhat different from the approach in \cite{kim2017heat}. If $\delta(t)=t^{\alpha}$ then  equation (\ref{main eqn}) has uniform ellipticity on $[e^{-n-1}, e^{-n}]$ for each $n\geq 1$ since  $t^{\alpha}\approx e^{-n\alpha}$.   Using a classical result for equations with uniform ellipticity one can get  some local  estimates on such intervals, and combining local estimates one can derive (\ref{eqn second}).  In this article, since our equation is not locally  elliptic, we do not follow the idea in \cite{kim2017heat}. Instead, we use a certain approximation method. We first assume $\delta(t)\geq \varepsilon >0$ and prove (\ref{eqn second})  with constant $N(d,p)$ independent of $\varepsilon>0$, then we take $\varepsilon\to 0$ for the general case.

This paper is organized as follows. In Section 2 we introduce some function spaces and our main results, Theorem \ref{main thm 1} and  Theorem \ref{main thm 2}.  Theorem \ref {main thm 1} is proved in Section 3, and Theorem \ref {main thm 2} is proved in Section 4.

We finish the introduction with  notation used in the article. 
\begin{itemize}

\item We use Einstein's summation convention throughout this paper. 

\item $\bN$ and $\bZ$ denote the natural number system and the integer number system, respectively.
As usual $\fR^{d}$
stands for the Euclidean space of points $x=(x^{1},...,x^{d})$.
 For $i=1,...,d$, multi-indices $\alpha=(\alpha_{1},...,\alpha_{d})$,
$\alpha_{i}\in\{0,1,2,...\}$, and functions $u(x)$ we set
$$
u_{x^{i}}=\frac{\partial u}{\partial x^{i}}=D_{i}u,\quad
D^{\alpha}u=D_{1}^{\alpha_{1}}\cdot...\cdot D^{\alpha_{d}}_{d}u.
$$
%

\item $C^\infty(\fR^d)$ denotes the space of infinitely differentiable functions on $\fR^d$. 
$\cS(\fR^d)$ is the Schwartz space consisting of infinitely differentiable and rapidly decreasing functions on $\fR^d$.
By $C_c^\infty(\fR^d)$,  we denote the subspace of $C^\infty(\fR^d)$ with the compact support.

\item  For $n\in \bN$ and $\cO\subset \fR^d$ and a normed space $F$, 
by $C(\cO;F)$,  we denote the space of all $F$-valued continuous functions $u$ on $\cO$ having $|u|_{C}:=\sup_{x\in O}|u(x)|_F<\infty$.

\item For $p \in [1,\infty)$, a normed space $F$,
and a  measure space $(X,\mathcal{M},\mu)$, 
by $L_{p}(X,\cM,\mu;F)$,
we denote the space of all $F$-valued $\mathcal{M}^{\mu}$-measurable functions
$u$ so that
\[
\left\Vert u\right\Vert _{L_{p}(X,\cM,\mu;F)}:=\left(\int_{X}\left\Vert u(x)\right\Vert _{F}^{p}\mu(dx)\right)^{1/p}<\infty,
\]
where $\mathcal{M}^{\mu}$ denotes the completion of $\cM$ with respect to the measure $\mu$. 
If there is no confusion for the given measure and $\sigma$-algebra, we usually omit them.

\item For  measurable set  $\cO \subset \fR^d$, $|\cO|$ denotes the Lebesgue measure of $\cO$.


\item By $\cF$ and $\cF^{-1}$ we denote the d-dimensional Fourier transform and the inverse Fourier transform, respectively. That is,
$\cF[f](\xi) := \int_{\fR^{d}} e^{-i x \cdot \xi} f(x) dx$ and $\cF^{-1}[f](x) := \frac{1}{(2\pi)^d}\int_{\fR^{d}} e^{ i\xi \cdot x} f(\xi) d\xi$.

\item If we write $N=N(a,b,\cdots)$, this means that the
constant $N$ depends only on $a,b,\cdots$. 
\end{itemize}




\mysection{Setting and main results}

Let $T \in (0,\infty)$ be a fixed time and $d \in \bN$ be the space dimension. 
\begin{assumption}
					\label{co as}
(i)  The coefficients $a^{ij}(t)$ $(i,j=1,\ldots,d)$ are measurable and bounded, that is there exists a constant $M>0$ such that 
$$
|a^{ij}(t)|\leq M, \quad \forall   \, t>0, i,j.
$$

(ii) There exists a nonnegative measurable function $\delta(t)$ such that 
 $$
0\leq \delta(t)|\xi|^2 \leq a^{ij}(t) \xi^i \xi^j, \quad \forall  (t,\xi) \in  (0,\infty) \times \fR^d.
 $$
\end{assumption}

\begin{remark}
\begin{enumerate}[(i)]
\item Since we may assume that $(a^{ij}(t))$ is symmetric, we can take $\delta(t)$ as the smallest eigenvalue of  $(a^{ij}(t))$. If $d=1$ we take $\delta(t)=a^{11}(t)$.

\item $\delta(t)$ is bounded due to Assumption \ref{co as}.
\end{enumerate}
\end{remark}

For $p \in  (1,\infty)$ and  $n=0,1,2,\cdots$, denote
$$
H^n_p=H^n_p(\fR^d)=\{u: D^{\alpha}u\in L_p(\fR^d), |\alpha|\leq n\},
$$
and in general for $n\in \fR$,  we write $u\in H^n_p$ iff
$$
\|u\|_{H^n_p}:=\|(1-\Delta)^{n/2}u\|_{L_p}:= \left\|\cF^{-1}[(1+|\xi|^2)^{n/2}\cF(u)(\xi)] \right\|_{L_p} <\infty.
$$
%
Define
$$
\bH^n_p(T)=L_p([0,T],dt;H^n_p)
$$
and for  $m  \in \fR \setminus 0$,
$$
\bH^{n}_p(T,\delta^m):=L_p( [0,T], \delta^m(t)dt ; H^{n}_p),
$$
i.e. 
\begin{align*}
f \in \bH^{n}_p(T,\delta^m)
\quad \Leftrightarrow \quad
\|f\|_{\bH^{n}_p(T,\delta^m)}:=  \left(\int_0^T \|f(t,\cdot)\|^p_{H^{n}_p} \delta^m(t)dt  \right)^{1/p}< \infty. 
\end{align*}
Simply we put
$$
\bL_p(T):=\bH^{0}_p(T), \quad \bL_p(T,\delta^m)=\bH^0_p(T,\delta^m).
$$
\begin{remark}
					\label{rmk sp}
(i) Let $u\in \bH^n_p(T,\delta^m)$ for some $m>0$. Then possibly $u(t,\cdot)\not\in H^n_p$ on the set $\{t: \delta(t)=0\}$.

(ii) Since $\delta(t)$ is bounded,  for any $m_1  \geq m_2$,
$$
\bH_p^\gamma(T,\delta^{m_2}) \subset \bH_p^\gamma(T, \delta^{m_1}),
$$
and in particular
$$
\bH_p^2(T) \subset \bH_p^2(T,\delta) \quad \text{and} \quad \bL_p(T,\delta^{1-p}) \subset \bL_p(T) .
$$
\end{remark}

\begin{definition}
						\label{sol def}
Let $u\in \bL_p(T)$. We say that $u$ is a solution to equation (\ref{main eqn})  iff    for any
 $\varphi \in C_c^\infty(\fR^d)$ the equality
\begin{align}
					\label{def sol}
&\left( u(t,\cdot), \varphi \right) = \left( u_0, \varphi \right)+\int_0^t \left( a^{ij}(s)u(s,\cdot),\varphi_{x^ix^j} \right) ds + \int_0^t \left(f(s,\cdot), \varphi \right) ds
\end{align}
holds for all $t\leq T$.
\end{definition}

Here is a  classical result for the fully degenerate parabolic equation. See e.g. \cite[Theorem 3.1]{krylov1986characteristics}, where it is assumed that $p \in [2,\infty)$ to handle stochastic parabolic equations. But for the deterministic case the proof goes through for all $p \in (1,\infty)$.

\begin{theorem}
     \label{classic}
Let $p>1$, $T<\infty$ and Assumption \ref{co as} hold.  Then for any $f\in \bL_p(T)$ and $u_0\in L_p$, equation (\ref{main eqn})
has a unique solution $u\in C\left( [0,T] ; L_p \right)$, and for this solution 
\begin{equation}
   \label{eqn classic}
\|u\|_{C([0,T]; L_p)}\leq N(p,T,d) \left(\|f\|_{\bL_p(T)}+\|u_0\|_{L_p} \right).
\end{equation}
Furthermore, if $f\in \bH^n_p(T)$ and $u_0\in H^n_p$ for some $n \in \fR$, then $u\in C([0,T]; H^n_p)$.
\end{theorem}

\begin{remark}
Note that in the above result, the regularity (or differentiability) of the solution is not better than that of $f$ and $u_0$.
\end{remark}

To state our regularity condition for the initial data, we introduce the Besov space characterized by the Littlewood-Paley operator. See \cite[Chapter 6]{bergh1976interpolation} or 
 \cite[Chapter 6]{grafakos2009modern} for more details.  Let $\Psi$ be a 
  nonnegative function on $\fR^d$ so that $\hat \Psi \in C_c^\infty \left( B_{2}(0)  \setminus B_{1/2}(0)\right)$
and
\begin{align}
                    \label{sum 1}
 \quad \sum_{j \in \bZ} \hat \Psi (2^{-j} \xi) = 1, \quad \forall \xi \in \fR^d,
\end{align}
where $B_r(0) := \{ x \in \fR^d : |x| \leq r\}$ and $\hat \Psi$ is the Fourier transform of $\Psi$. 
For a tempered distribution $u$, we define
\begin{align}
				\label{del j}
\Delta_j u(x):=\Delta_j^\Psi u(x):= \cF^{-1} \left[ \hat \Psi (2^{-j} \xi) \cF u (\xi) \right] (x)
\end{align}
and
$$
S_0(u)(x) = \sum_{j=-\infty}^0 \Delta_j u (x),
$$
where the convergence is understood in the sense of distributions.
Due to \eqref{sum 1},
\begin{align}
                    \label{little iden}
u(x)= S_0(u)(x) + \sum_{j=1}^\infty \Delta_j u(x).
\end{align}
The Besov space $B^s_p$ with the order $s$ and the exponent $p$  is  the space of all tempered distributions $u$ such that
\begin{equation}
    \label{Besov}
\|u\|_{B_p^s}:=\| S_0(u) \|_{L_p} + \left(\sum_{j=1}^\infty 2^{spj} \| \Delta_j u \|_{L_p}^p \right)^{1/p} < \infty.
\end{equation}

Now we introduce our first result to equation (\ref{main eqn}).  We control an $L_p$-norm of $u_{xx}$  with the help of the weight $\delta(t)dt$. 
\begin{theorem}
					\label{main thm 1}
Let $p>1$, $n\in \fR$,  and Assumption \ref{co as} hold. Then for all $u_0 \in B_p^{n+2-2/p}$
and $f \in \bH^n_p \left(T, \delta^{1-p} \right)$, there exists a unique solution 
$u \in C\left( [0,T] ; H^n_p \right) \cap \bH^{n+2}_p(T,\delta)$ to the  problem
 \begin{align}
					\label{eqn 2.16.1}
u_t=a^{ij}(t)u_{x^ix^j}+f, \quad t \in (0,T)\, ; \quad u(0,\cdot)=u_0.
\end{align}
Furthermore, for this solution  we have
\begin{equation}
                             \label{main-zero}
 \|u_{xx}\|_{\bH^n_p(T,\delta)}
\leq N(d,p) \left(\|f\|_{\bH^n_p(T, \delta^{1-p})} 
+\|u_0\|_{B_p^{n+2-2/p}}\right).
 \end{equation}
\end{theorem}

The proof of Theorem \ref{main thm 1} is given in Section \ref{pf main thm 1} (see Lemma \ref{opt est}).

\begin{remark}
(i) By (\ref{main-zero}), $u_{xx}(t,\cdot)\in L_p$ on the set $\{t: \delta(t)>0\}$.

(ii) We do not assume $\int_0^T \delta(s) ds > 0$ in Theorem \ref{main thm 1}.
In other words,  the equation can be completely  degenerate on $(0,T)$. 
However, if $\delta(t)=0$ for almost every $t \in (0,T)$, then (\ref{main-zero}) becomes trivial. 
\end{remark}

In the following remark we consider the general case having coefficients depending on $(t,x)$ and uniformly continuous in $x$.

\begin{remark}
   \label{variable x}
  For simplicity assume $u_0=0$. Let $(a^{ij})=(a^{ij}(t,x))\geq \delta(t)I_{d\times d}$.  Then, using (\ref{main-zero})  and a standard freezing  coefficients argument one can easily prove  that there exists $\varepsilon_0>0$ such that (\ref{main-zero}) holds with $n=0$ if $|a^{ij}(t,x)-a^{ij}(t,y)|\leq \varepsilon_0 \delta(t)$ for all $x,y$ in the support of $u$. 
  
Now assume  that  $|a^{ij}(t,x)|\leq N\delta(t)$ and $a^{ij}(t,x)$ are uniformly continuous in $x$ in  the sense that  for any  $\varepsilon>0$, there exists $\kappa>0$, independent of $t$, such that  
   \begin{equation}
      \label{eqn 1}
 |x-y|<\kappa \quad \Rightarrow \quad   |a^{ij}(t,x)-a^{ij}(t,y)|\leq \varepsilon \delta(t).
   \end{equation}
   Note that (\ref{eqn 1}) holds if $a^{ij}(t,x)=a^{ij}(t)\eta(x)$, where $\eta(x)$ is nonnegative bounded uniformly continuous function.

   Next we onsider an appropriate partition of unity $\{\eta^n: n\geq 1\}$ of $\fR^d$ such that each $\eta^n$  has a support in a ball of radius $\kappa_0$ which corresponds to $\varepsilon_0$ in (\ref{eqn 1}). Note that for each $n$,
   $$
   (u\eta^n)_t=a^{ij}(u\eta^n)_{x^ix^j}-2a^{ij}u_{x^i}\eta^n_{x^j}-a^{ij}u\eta^n_{x^ix^j}+f\eta^n, \quad t>0.
   $$
     Thus, by the choice of $\kappa_0$, for each $t\leq T$,
     \begin{equation}
      \label{eqn 6.21}
     \|(u\eta^n)_{xx}\|^p_{\bL_p(t,\delta)}\leq N \|-2a^{ij}u_{x^i}\eta^n_{x^j}-a^{ij}u\eta^n_{x^ix^j}+f\eta^n\|^p_{\bL_p(t,\delta^{1-p})}.
     \end{equation}
Obviously,
     $$
     \|\eta^n u_{xx}\|^p_{\bL_p(t,\delta)}\leq N\left( \|(u\eta^n)_{xx}\|^p_{\bL_p(t,\delta)}+\|u_x \eta^n_x\|^p_{\bL_p(t,\delta)}+\|u\eta^n_{xx}\|^p_{\bL_p(t,\delta)}\right).
     $$
Thus,  summing up these estimates with respect to $n$ and using (\ref{eqn 6.21}) we get
       \begin{align}
							\notag
\|u_{xx}\|^p_{\bL_p(t,\delta)}
&\leq  N\|u_x\|^p_{\bL_p(t, \delta)}+N\|u\|^p_{\bL_p(t,\delta)}
  +N\|a^{ij}u_{x^i}\|^p_{\bL_p(t,\delta^{1-p})}\\
							\label{2018061311}
  &+N\|a^{ij}u\|^p_{\bL_p(t,\delta^{1-p})}+N\|f\|^p_{\bL_p(T,\delta^{1-p})}.
  \end{align}
  This and the assumption $|a^{ij}(t,x)|\leq N\delta(t)$ certainly lead to
  $$
  \|u_{xx}\|^p_{\bL_p(t,\delta)}\leq  N\|u_x\|^p_{\bL_p(t, \delta)}+N\|u\|^p_{\bL_p(t,\delta)}
+N\|f\|^p_{\bL_p(T,\delta^{1-p})}.
  $$
 Furthermore, using inequalities 
\begin{align}
						\label{2018061501}
\|u_x(s,\cdot)\|_{L_p} \leq \epsilon
 \|u_{xx}(s,\cdot)\|_{L_p} +\epsilon^{-1} \|u(s,\cdot)\|_{L_p}, \quad \|u\|_{\bL_p(t,\delta)}\leq N\|u\|_{\bL_p(t)},
\end{align}
and taking $\epsilon$ sufficiently small, we get for each $t\leq T$  \begin{align}
  \label{2018061320}
  \|u_{xx}\|^p_{\bL_p(t,\delta)}\leq  N \left(\|u\|^p_{\bL_p(t)} +\|f\|^p_{\bL_p(T,\delta^{1-p})}\right).
  \end{align}
Moreover, since  $u$ is a solution to \eqref{eqn 2.16.1} and $u_0=0$, for all $s \leq t\leq T$, we get 
\begin{align}
						\notag
\|u(s,\cdot)\|^p_{L_p} 
&= \left\| \int_0^s u_r(r,\cdot) dr \right\|^p_{L_p} = \left\| \int_0^s (a^{ij}u_{x^ix^j}+f)(r,\cdot) dr \right\|^p_{L_p} \\
						\notag
&\leq N \int_0^s  \left( \delta^p(r)\|u_{xx}(r,\cdot)\|^p_{L_p}  + \|f(r,\cdot) \|^p_{L_p}\right) dr \\
						\label{2018061301}
&\leq N\int_0^s  \left( \delta(r)\|u_{xx}(r,\cdot)\|^p_{L_p}  + \|f(r,\cdot) \|^p_{L_p}\right) dr ,
\end{align}
where the last inequality is due to the assumption that $\delta(t)$ is bounded.
Inequality (\ref{2018061301}) and  integration  on $[0,t]$ give 
\begin{align}
							\label{2018061321}
\|u\|^p_{\bL_p(t)} \leq  N(T)  \left( \int_0^t \|u_{xx}\|^p_{\bL_p(s,\delta)} ds + \|f\|^p_{\bL_p(T,\delta^{1-p})}\right), \quad \forall \, t\leq T.
\end{align}
From \eqref{2018061320} and \eqref{2018061321},  we get, for any $t\leq T$,
 \begin{align}
								\notag
\|u_{xx}\|^p_{\bL_p(t,\delta)} &\leq N \|u\|^p_{\bL_p(t)}+N\|f\|^p_{\bL_p(T,\delta^{1-p})}\\
								\label{2018061523}
  &\leq N \int^t_0 \|u_{xx}\|_{\bL_p(s,\delta)}^p ds+ N\|f\|^p_{\bL_p(T,\delta^{1-p})}.
  \end{align}
This and Gronwall's inequality lead to (\ref{main-zero}) with $n=0$.

  Finally,  we  explain that  one can slightly weaken the condition $|a^{ij}(t,x)|\leq N \delta(t)$ by 
  \begin{equation}
    \label{eqn 6.21.5}
  |a^{ij}(t,x)|\leq N  \left(\delta(t) \right)^{\frac{2p-1}{2p}},  \quad p \geq \frac{3}{2}.
  \end{equation}         
Note that  from \eqref{2018061501}  we can obtain
   \begin{align}
	\label{2018061510}						
\|u_x\|_{\bL_p(t, \sqrt{\delta})} 
\leq  \frac{1}{\epsilon}\|u\|_{\bL_p(t)} + \epsilon \|u_{xx}\|_{\bL_p(t,\delta) },
\end{align}
where $\epsilon$ is a arbitrary positive constant. 
From \eqref{2018061311} and (\ref{eqn 6.21.5}),
  \begin{align}
							\notag
\|u_{xx}\|_{\bL_p(t,\delta)}
&\leq  N\|u_x\|_{\bL_p(t, \delta)}+N\|u\|_{\bL_p(t,\delta)}
  +N\|a^{ij}u_{x^i}\|_{\bL_p(t,\delta^{1-p})}\\
							\notag
  &\quad +\|a^{ij}u\|_{\bL_p(t,\delta^{1-p})}+\|f\|_{\bL_p(T,\delta^{1-p})}  \\
						\notag
&\leq  N\|u_x\|_{\bL_p(t, \sqrt{\delta})}+N\|u\|_{\bL_p(t)} +N\|u_{x}\|_{\bL_p(t, \sqrt{\delta})}+\|u\|_{\bL_p(t,\sqrt{\delta})} \\
								\notag
  &\quad +\|f\|_{\bL_p(T,\delta^{1-p})}  \\
									\notag
  &\leq  N\|u_x\|_{\bL_p(t, \sqrt{\delta})}+N\|u\|_{\bL_p(t)}+\|f\|_{\bL_p(T,\delta^{1-p})}\\
  &\leq \frac{N}{\epsilon}\|u\|_{\bL_p(t)} + N \epsilon \|u_{xx}\|_{\bL_p(t,\delta)}+N\|u\|_{\bL_p(t)}+\|f\|_{\bL_p(T,\delta^{1-p})} .
  \end{align}
  Taking $\epsilon$ sufficiently small, we get  (\ref{2018061320}) again.  
Moreover, following \eqref{2018061301}, we have
\begin{align}
						\notag
\|u(s,\cdot)\|^p_{L_p} 
&= \left\| \int_0^s u_r(r,\cdot) dr \right\|^p_{L_p}  \\
						\notag
&\leq N\int_0^s  \left( \delta^{\frac{2p-1}{2}}(r)\|u_{xx}(r,\cdot)\|^p_{L_p}  + \|f(r,\cdot) \|^p_{L_p}\right) dr \\
						\label{2018061530}
&\leq  N \int_0^s  \left( \delta(r)\|u_{xx}(r,\cdot)\|^p_{L_p}  + \|f(r,\cdot) \|^p_{L_p}\right) dr ,
\end{align}
Therefore \eqref{2018061523} is obtained again.
   \end{remark}

\vspace{2mm}

In Assumption \ref{ini as} below we assume certain asymptotic behavior of $\delta(t)$ near $t=0$ 
to obtain
\begin{equation}
      \label{eqn asym}
\|u\|_{\bL_p(T)}+\|u_{xx}\|_{\bL_p(T)}\leq C \|u_0\|_{B^s_p},
\end{equation}
where $s<2$ and $u$ is the solution to equation (\ref{main eqn}) with $f=0$. Obviously, (\ref{eqn asym}) is impossible if $\delta(t)$ completely vanishes near $t=0$.  Indeed, if $a^{ij}(t)\equiv 0$ near $t=0$ then we get $u=u_0$ near $t=0$.

\begin{assumption}
					\label{ini as}
\begin{enumerate}[(i)]
\item  $\delta(t)$ does not completely vanish near $t=0$. In other words,  $\int_0^t \delta(s) ds >0$ for all $t>0$.

\item
There exist $t_0 \in (0,T)$,  $\beta>0$, and $N_0>0$ such that for all $h >0$, 
\begin{align}
					\label{0427 e 1}
\left| \left\{ t \in [0,t_0] : h \leq \int_0^t \delta (s) ~ds <4h \right\} \right| \leq N_0 h^{1/\beta}.
\end{align}
\item There exists a constant $  \bar N_0>0$ such that
\begin{equation}
    \label{1d}
     |a^{ij}(t)| \leq \bar N_0 \delta(t) \qquad \forall t>0. 
\end{equation}

\end{enumerate}
\end{assumption}
\begin{remark}
\begin{enumerate}[(i)]

\item 
If $d=1$ then we can take $\delta(t)=a^{11}(t)$, and thus (\ref{1d}) holds with $\bar{N}_0=1$.

\item
If \eqref{0427 e 1}  holds with some $t_0>0$, then it also holds for any $t'_0<t_0$ and therefore we may assume that $t_0$ is very small. 
In particular, we put $t_0 <1 \wedge T$. 

\item
Since $t_0$ can be taken very small and $\int_0^t \delta(s)ds$ goes to zero as $t \to 0$, 
it is sufficient that $(\ref{0427 e 1})$ holds only for all sufficiently small $h>0$. 


\item Obviously, if  $\delta(t)>c>0$ near $t=0$ then  \eqref{0427 e 1}   holds with $\beta=1$.
\end{enumerate}
\end{remark}
\bigskip

Here are two examples related to   Assumption \ref{ini as}.



\begin{example}[Functions with weak scaling property]
						\label{ex 3}
Assume that there exist  constants $\alpha >-1$, $t_0>0$, and  $N>0$ such that
\begin{align}
						\label{weak sca}
N t^{\alpha} \leq \delta(t) \qquad \forall t \in (0,t_0).
\end{align}
Then  Assumption \ref{ini as}(ii) holds with $\beta=\alpha+1$. Indeed,
\begin{align*}
\left| \left\{ t \in [0,t_0] : h \leq \int_0^t \delta (s) ~ds <4h \right\} \right| 
&\leq \left| \left\{ t \in [0,t_0] : \frac{ N}{\alpha+1} t^{\alpha+1}  <4h \right\} \right|  \\
&\leq N_0(N,\alpha)  h^{1/(\alpha+1)}.
\end{align*}
(\ref{weak sca}) holds if $\delta(t) (\geq 0)$ is a polynomial or analytic near zero.  Here are other (Bernstein)  functions satisfying (\ref{weak sca}):  
\begin{itemize}
\item[(1)] $\delta(t)=\sum_{i=1}^n t^{\alpha_i}$, $0<\alpha_i  <1$;
\item[(2)] $\delta(t)=(t+t^\alpha)^\beta$, $\alpha, \beta\in (0, 1)$;
\item[(3)] $\delta(t)=t^\alpha(\log(1+t))^\beta$, $\alpha\in (0, 1)$,
$\beta\in (0, 1-\alpha)$;
\item[(4)] $\delta(t)=t^\alpha(\log(1+t))^{-\beta}$, $\alpha\in (0, 1)$,
$\beta\in (0, \alpha)$;
\item[(5)] $\delta(t)=(\log(\cosh(\sqrt{t})))^\alpha$, $\alpha\in (0, 1)$;
\item[(6)] $\delta(t)=(\log(\sinh(\sqrt{t}))-\log\sqrt{t})^\alpha$, $\alpha\in (0, 1)$.
\end{itemize}
\end{example}

\begin{example}[Oscillatory functions] 
						\label{ex 4}
Assume that
there exist  constants $\beta_0 > 0$, $t_0  > 0$, and $N>0$ such that
\begin{align}
						\label{osc con}
N t^{\beta_0} \leq  \int_0^t\delta(s) ds \qquad \forall t \in (0,t_0).
\end{align}
Then by the argument in the previous example, $\delta(t)$ satisfies  Assumption \ref{ini as}(ii) with $\beta=\beta_0$. Condition (\ref{osc con}) is a 
generalization of (\ref{weak sca}) and is satisfied by  lots of interesting oscillatory functions.
For example, put 
$$\delta(t) = 1+ \sin (1/t).$$ 
Note that $\delta(t)$ vanishes infinitely many times near $t=0$,  and surprisingly (\ref{osc con}) holds with $\beta_0=1$. This is because for any small $t>0$, 
$$|A_t|:=|\{s\leq t: \sin(1/s)\geq -1/2\}|\geq t/2,$$
and therefore $\int^t_0 (1+\sin(1/s))ds \geq \int_{A_t} 1/2 ds \geq t/4$.

\end{example}

\begin{theorem}
					\label{main thm 2}
Let $p>1$ and $T<\infty$. 
Suppose that  Assumptions \ref{co as} and \ref{ini as} hold. 
Then, for any
 $u_0 \in B^{2 \left(1-1/(\beta p) \right)}_p$, there exists a unique solution 
$u \in C\left( [0,T] ; L_p \right) \cap \bH^2_p(T)$ to the problem
 $$
u_t=a^{ij}(t)u_{x^ix^j}, \quad t \in (0,T)\, ; \quad u(0,\cdot)=u_0,
$$
and for this solution we have
\begin{align}
					\label{0510 e 1}
\|u_{xx}\|_{\bL_p(T)} \leq 
N \|u_0\|_{B_p^{2 \left(1-1/(\beta p) \right)}},
\end{align}
where  $N$ is a  constant depending only on $d,p,T,N_0, \bar N_0, \beta$, and $ \int_0^{t_0} \delta(s)ds$. 
\end{theorem}
The proof of this theorem will be given in Section \ref{pf main thm 2}.

We remark that  unlike   (\ref{main-zero}), estimate (\ref{0510 e 1})  control $u_{xx}$ without the help of the weight $\delta(t)dt$.  
\begin{remark}
					\label{initial remark}
If $a^{ij}(t) = \delta^{ij}$, where
$\delta^{ij}$ is the Kronecker delta, then (\ref{0510 e 1}) is a classical result for the heat equation.
In this case, the constant $N$ in (\ref{0510 e 1})  depends  only on $d$ and $p$, and in particular it is independent of $T$. This  can be easily checked by the standard scaling argument 
(in homogeneous Besov space). 
\end{remark}

Example \ref{ex 4} and Theorem \ref{main thm 2} yield the following result.

\begin{corollary}
Let $p \in (1,\infty)$, $T \in (0,\infty)$, $u_0 \in B^{2-2/p}_p$, and $d=1$. Then there exists a unique solution 
$u  \in  \bH^2_p(T)$
to  the equation 
\begin{equation}
      \label{deg}
u_t=(1+\sin(1/t))u_{xx}, \quad t>0; \quad u(0,\cdot)=u_0,
\end{equation}
and we have
\begin{equation}
    \label{eqn maximal}
\|u\|_{\bH^2_p(T)}\leq N\|u_0\|_{B^{2(1-1/p)}_p}.
\end{equation}
\end{corollary}
Recall that (\ref{eqn maximal}) is the maximal regularity of the solution to the heat equation $u_t=u_{xx}$.  Thus, the instant smoothing effect (or regularity of solution) of  degenerate  equation (\ref{deg})  is not affected at all by the degeneracy of the leading coefficient.

\mysection{Proof of Theorem \ref{main thm 1}}
							\label{pf main thm 1}

Denote $A(t):=\left(a^{ij}(t) \right)$. Recall in Theorem \ref{main thm 1}, we only assume
$$
0\leq \delta(t) I_{d\times d} \leq A(t), \quad |a^{ij}(t)|\leq M.
$$

 \begin{lemma}
    \label{lemma classic}
Let $\delta(t)\geq \varepsilon >0$, $u_0 \in B_p^{2 \left(1-1/ p \right)}$, $f\in \bL_p(T)$, and $u\in \bL_p(T)$ be a solution to problem (\ref{eqn 2.16.1}), that is
$$
u_t=a^{ij}u_{x^ix^j}+f, \quad t \in (0,T); \quad u(0,\cdot)=u_0.
$$
Then
 \begin{align}
					\label{constant}
 \|u_{xx}\|_{\bL_p(T)}
 \leq N(d,p) \left( \varepsilon^{-1} \|f\|_{\bL_p(T)} + \|u_0\|_{B_p^{2 \left(1-1/ p \right)}}\right).
 \end{align}
  \end{lemma}
  The key point of Lemma \ref{lemma classic} is that   the constant $N(d,p)$ in (\ref{constant}) is independent of $T$ and $M$.     Lemma \ref{lemma classic} might be a very well known result, but we provide a (probabilistic) proof for the sake of the completeness.

  \begin{proof}
  
  We follow the idea in the proof of  \cite[Theorem 2.2]{Krylov1994}.
Let $W_t=(W^1_t,\cdots, W^d_t)$ be a $d$-dimensional Wiener process on a probability space $(\Omega,\cF,P)$. 
Since  $A(t)$ is a nonnegative symmetric matrix, there exists a nonnegative symmetric matrix $\sigma(t)=(\sigma^{ij}(t))$ such that
$$
2A(t)= \sigma^2(t).
$$
We define 
\begin{align}
							\label{pdf x}
X_t := \int_0^t \sigma(t)dW_t,  \quad (i.e.,\, X^i_t=\sum_{k=1}^d \int^t_0 \sigma^{ik}(s)dW^k_s, \, (i=1,2,\cdots,d)).
\end{align}
It is well known  (see e.g. \cite{Krylov1994}) that the solution to equation (\ref{eqn 2.16.1}) is given by

\begin{align}
						\label{sol rep}
u(t,x)
= \bE[u_0(x+X_t)] 
+ \int_0^t\bE[f(s,x+X_t-X_s ] ds.
\end{align}
In particular if $A(t)=\varepsilon I_{d\times d}$, where $I_{d\times d}$ is the $d \times d$ identity matrix and $\varepsilon>0$, then  $X_t=\sqrt{2\varepsilon}W_t$,  and 
$$
u(t,x)
=\bE[u_0(x+\sqrt{2 \varepsilon} W_t)] 
+\int_0^t\bE[f(s,x+\sqrt{2\varepsilon}W_{t}-\sqrt{2\varepsilon}W_{s}] ds.
$$

\vspace{3mm}

{\bf Step 1}.   Assume  $\delta(t)=\varepsilon$ and $A(t)=\varepsilon I_{d\times d}$.  Then $v(t,x)=u(\varepsilon^{-1}t, x)$ satisfies the heat equation 
$$
v_t(t,x)=\Delta v(t,x)+\varepsilon^{-1}f(\varepsilon^{-1}t,x).
$$
From the classical result for the heat equation (cf. Remark \ref{initial remark}), it follows that 
$$
\|u_{xx}\|_{\bL_p(T)}
\leq N(d,p) \left( \varepsilon^{-1} \|f\|_{\bL_p(T)} + \|u_0\|_{B_p^{2 \left(1-1/ p \right)}}\right),
$$
and equivalently, we get
\begin{align}
					\notag
&\left\|D^2_x \left[\bE[u_0(x+\sqrt{2 \varepsilon} W_t)]
+\int_0^t\bE[f(s,x+\sqrt{2\varepsilon}W_{t}-\sqrt{2\varepsilon}W_{s} ] ds \right] \right\|_{\bL_p(T)}  \\
					\label{0503 e 2}
&\leq N(d,p) \left( \varepsilon^{-1} \|f\|_{\bL_p(T)} + \|u_0\|_{B_p^{2 \left(1-1/ p \right)}}\right).
\end{align}
\smallskip

{\bf Step 2}. General case. 
  Write
  $$
  A(t)= \left(A(t)-\frac{\varepsilon}{2}I_{d\times d} \right)+\frac{\varepsilon}{2}I_{d\times d}=: \bar{A}(t)+\frac{\varepsilon}{2}I_{d\times d}.
  $$
Let $\bar{W}_t$  be  a  $d$-dimensional Winer process which is independent of $W_t$ on a probability space $(\Omega,\cF,P)$. Denote 
  $$
  Y_t=\sqrt{\varepsilon}W_t, \quad \bar \sigma(t)=\sqrt{2\bar{A}(t)}, \quad Z_t=\int^t_0 \bar{\sigma}_s d\bar{W}_s.
  $$  
  Then it is easy to show that for $0\leq s<t$, $X_t-X_s$ and $Y_t+Z_t-Y_s-Z_s$ have the same probability distribution. 
  Indeed, it suffices to check that the characteristic functions of two random variables coincide (cf. \cite[Theorem 1.4.12]{Krylov1995}) and   for  any $\xi \in \fR^d$, we get
\begin{eqnarray*}
&&\bE e^{i \xi \cdot (Y_t-Y_s +Z_t-Z_s)} =\bE e^{i \xi \cdot (Y_t-Y_s)} \bE e^{i\xi \cdot (Z_t-Z_s)}
=e^{-\varepsilon (t-s)|\xi|^2} \bE e^{i \xi \cdot \int^t_s \bar{\sigma}_s d\bar{W}_r}\\
&=&e^{-\varepsilon (t-s)|\xi|^2} e^{-\int^t_s |\bar{\sigma}_r \xi|^2 dr}=e^{-\int^t_s |\sigma_r \xi|^2 dr}=\bE e^{i\xi \cdot \int^t_s \sigma_r dW_r}.
 \end{eqnarray*}
The third and fifth equalities above  are trivial if $\bar{\sigma}_t$ is a simple function, and  the general case is obtained based on a standard approximation.

  From (\ref{sol rep}) it follows that  (recall that $Y$ and $Z$ are independent)
  \begin{align*}
  D^2_xu(t,x) 
   &= D^2_x \left(\bE[u_0(x+ Y_t + Z_t)] + \int^t_0 \bE \left[f(s,x+Y_{t}-Y_{s}+Z_{t}-Z_{s}) \right]ds \right)\\
  &=D^2_x \left(  \bE' \left( \bE \left[u_0(x+Y_{t}(\omega)+Z_{t}(\omega')) \right] \right) \right) \\
  &\quad +D^2_x \left( \int^t_0 \bE' \left( \bE \left[f(s,x+Y_{t}(\omega)-Y_{s}(\omega)+Z_{t}(\omega')-Z_{s}(\omega')) \right] \right)ds \right),
  \end{align*}
where $\bE$ and $\bE'$ denote the expectations with respect to $\omega$ and $\omega'$, respectively. 
For each $\omega' \in \Omega$, denote 
  $$
 f_{Z(\omega')}(s,x)=f(s,x-Z_s(\omega')).
  $$
Then,
  \begin{align*}
  D^2_xu(t,x)
  &=\bE' \left[ D^2_x \left(   \bE \left[u_0(x+\sqrt{\varepsilon}W_{t}(\omega)+Z_{t}(\omega')) \right] \right) \right] \\
  &\quad +\bE' \left[D^2_x \left(\int^t_0 \bE \left[f_{Z(\omega')}(s,x+\sqrt{\varepsilon}W_{t}(\omega)-\sqrt{\varepsilon}W_{s}(\omega)+Z_t(\omega')) \right] ds \right) \right].
  \end{align*}
Finally, since the $L_p(\fR^d)$-norm is translation invariant,  by (\ref{0503 e 2}) we get
  \begin{align*}
& \|D^2_xu\|^p_{\bL_p(T)} \\
&\leq \bE' \Bigg[ \bigg\| 
D^2_x \left(    \bE \left[u_0(x+\sqrt{\varepsilon}W_{t}(\omega)+Z_{t}(\omega')) \right]  \right) \\
&\qquad +D^2_x \left( \bE \int^t_0 f_{Z(\omega')}(s,x+\sqrt{\varepsilon}W_{t}(\omega)-\sqrt{\varepsilon}W_{s}(\omega)+Z_t(\omega')) ds \right) \bigg\|^p_{\bL_p(T)} \Bigg] \\
&=\bE' \Bigg[ \bigg\| 
D^2_x \left(    \bE \left[u_0(x+\sqrt{\varepsilon}W_{t}(\omega)) \right]  \right)   \\
&\qquad + D^2_x \left[\bE \int^t_0 f_{Z(\omega')}(s,x+\sqrt{\varepsilon}W_{t}(\omega)- \sqrt{\varepsilon}W_{s}(\omega)) ds \right]\bigg\|^p_{\bL_p(T)} \Bigg] \\
&\leq 
N(d,p)  \bE' \left[ \varepsilon^{-1}  \|f_{Z(\omega')}\|^p_{\bL_p(T)}  + \|u_0\|_{B_p^{2 \left(1-1/ p \right)}}\right] \\
&=N(d,p)  \left[ \varepsilon^{-1}  \|f\|^p_{\bL_p(T)}  + \|u_0\|_{B_p^{2 \left(1-1/ p \right)}}\right].
  \end{align*}
The lemma is proved. 
   \end{proof}

 Note that to prove Theorem \ref{main thm 1} it is enough to assume $n=0$. Thus we only need to prove the following.
 
 \begin{lemma}
					\label{opt est}
Let $p>1$, $T<\infty$, $f\in \bL_p(T,\delta^{1-p})$, and $u \in C([0,T];L_p)$ be a solution to equation (\ref{eqn 2.16.1}). 
Then 
\begin{align*}
 \int^T_0 \|u_{xx}(t)\|^p_{L_p} \delta(t)dt 
 \leq N(p,d) \left(\int^T_0 \|f(t)\|^p_{L_p}\delta^{1-p}(t) dt + \|u_0\|^p_{B_p^{2 \left(1-1/ p \right)}}\right).
 \end{align*}
\end{lemma}
\begin{proof}

 {\bf{Step 1}}.  In this step, we assume $\delta(t) \geq \varepsilon>0$. Then
 $$
  \int^{\infty}_0 \delta(t)dt=\infty,   \quad \text{and}\quad  \delta^{-1}\in L_1([0,T]).
 $$
Denote
$$
\beta(t)=\int^t_0 \delta(s)ds,
$$
and let $\phi(t)$ be the inverse of $\beta(t)$, which is well defined for $t\in [0,\infty)$ due to $\delta(t)>\varepsilon>0$.  Then, since $\beta'(t)>0$, $\phi$ is differentiable everywhere and 
$$
\phi'(t)=\frac{1}{\beta'(\phi(t))}=\frac{1}{\delta(\phi(t))}\leq \varepsilon^{-1}.
$$
Thus, in particular, $\phi$ is absolutely continuous.
 Define
 $$
 v(t,x)=u(\phi(t),x).
 $$
 Note that $v$ satisfies 
 $$
 v_t=\tilde{a}^{ij}(t)v_{x^ix^j} +\tilde{f}, \quad v(0,\cdot)= u_0
 $$
 where $\tilde{f}(t,x)=f(\phi(t),x)\phi'(t)$ and
 $$
 \tilde{a}^{ij}(t):=a^{ij}(\phi(t))\phi'(t)=\frac{1}{\delta(\phi(t))} a^{ij}(\phi(t))\geq I_{d\times d}.
 $$
 Thus by Lemma \ref{lemma classic},
 for any $T_0>0$ such that $\phi(T_0) \leq T$, we have
 $$
 \|v_{xx}\|_{\bL_p(T_0)}\leq N(d,p) \left( \|\tilde{f}\|_{\bL_p(T_0)} +\|u_0\|_{B_p^{2 \left(1-1/ p \right)}}\right).
 $$
Taking $T_0>0$ so that $\phi(T_0)=T$, we have
 \begin{equation}
				 \label{deterministic}
\int^T_0 \|u_{xx}\|^p_{L_p}\delta(t)dt 
\leq N(d,p) \left(\int^T_0 \|\delta^{-1}(t)f\|^p_{L_p} \delta(t) dt+ \|u_0\|_{B_p^{2 \left(1-1/ p \right)}}\right).
\end{equation}

 {\bf{Step 2}}.  Let $\delta(t)\geq 0$, and assume  $f\in \bH^2_p(T)$ and $u_0\in H^2_p$. Then by Theorem \ref{classic},
 \begin{equation}
    \label{high}
 u\in \bH^2_p(T).
 \end{equation}
  Denote 
 $$A_{\varepsilon}(t)=A(t)+\varepsilon I, \quad \delta_{\varepsilon}(t):=\delta(t)+\varepsilon.
 $$
 Then 
 $$
 u_t=a^{ij}_{\varepsilon}u_{x^ix^j}+f-\varepsilon \Delta u, \quad u(0,\cdot)=u_0.
 $$
 By Step 1 and the fact that $\delta^{-p+1}_{\varepsilon}\leq \delta^{-p+1}$,
 \begin{align*}
&\int^T_0 \|u_{xx}\|^p \delta_{\varepsilon}(t) dt  \\
&\leq N(p,d) \left(\int^T_0 \|\delta^{-1}_{\varepsilon} (f-\varepsilon \Delta u)\|^p_{L_p} \delta_{\varepsilon}(t) dt 
+\|u_0\|_{B_p^{2 \left(1-1/ p \right)}}\right)\\
 &\leq N(p,d) \left(\int^T_0 \|\delta^{-1} f\|^p_{L_p}\delta(t)dt+\|u_0\|_{B_p^{2 \left(1-1/ p \right)}}\right)
 + N(p,d)  \int^T_0 \|\delta^{-1}_{\varepsilon} \varepsilon \Delta u\|^p_{L_p} \delta_{\varepsilon}(t) dt.
 \end{align*}
 Note that $\frac{\varepsilon^p}{\delta^{p}_{\varepsilon}} \delta_{\varepsilon}$ is bounded above by $\delta+\varepsilon$ and goes to zero as $\varepsilon \to 0$. This is because if $\delta(t)=0$ then $\frac{\varepsilon^p}{\delta^{p}_{\varepsilon}} \delta_{\varepsilon} =\varepsilon$.  
Moreover, $\Delta u \in \bL_p(T)$ due to (\ref{high}).
 Therefore, by the dominated convergence theorem and the inequality $\delta\leq \delta_{\varepsilon}$, we get
 \begin{eqnarray}
      \label{final}
  \int^T_0 \|u_{xx}\|^p \delta(t) dt 
  \leq N(p,d) \left( \int^T_0 \|\delta^{-1} f\|^p_{L_p}\delta(t)dt +\|u_0\|_{B_p^{2 \left(1-1/ p \right)}}\right).
  \end{eqnarray}

 {\bf{Step 3}}. In general, we consider a mollification with respect to the space variable.
 $$
 u^{\varepsilon}_t=a^{ij}(t)u^{\varepsilon}_{x^ix^j}+f^{\varepsilon}, \quad u^\varepsilon(0,\cdot) = u_0^\varepsilon
 $$
where 
$$
u^\varepsilon(t,x) = \varepsilon^{-d} \int_{\fR^d} u(t,x-y) \varphi(\varepsilon^{-1} y) dy
$$
 and $\varphi \in C_c^\infty(\fR^d)$ with the unit integral. 
 Then $f^{\varepsilon}\in \bH^2_p(T)$ and $u^{\varepsilon}_0\in H^2_p$, and by Theorem \ref{classic}  $u^{\varepsilon}\in \bH^2_p(T)$. By Step 2,  it follows that $u^{\varepsilon}$ is a Cauchy sequence in $L_p([0,T], \delta (t)dt ; H^2_p)$.  The limit $u$ certainly satisfies the equation and  estimate (\ref{final}) also holds for $u$.  
The lemma is proved. 
  \end{proof}
\bigskip

%
%

\mysection{Proof of Theorem \ref{main thm 2}} 
 										\label{pf main thm 2}

 Recall that in Theorem \ref{main thm 2} we assume
 $$
 \int^t_0 \delta(s)ds>0, \quad \forall\, t>0.
 $$
 
By taking the Fourier transform to the equation
\begin{equation}
    \label{eqn homo-2}
u_t=a^{ij}(t)u_{x^ix^j}, \quad t \in (0,T); \quad u(0,\cdot)=u_0,
\end{equation}
we get (at least formally)
\begin{align*}
\cF[u(t,\cdot)](\xi)= 
\exp \left( -\int_0^t a^{ij}(r) dr  \xi^i \xi^j \right) \cF[u_0](\xi),
\end{align*}
and the inverse Fourier transform gives
\begin{equation}
    \label{formula}
u(t,x) = \int_{\fR^d} p(t,x-y)u_0(y)dy =p(t, \cdot) \ast u_0(x)  =:\cT_t u_0(x),
\end{equation}
where 
\begin{align}
					\label{0505 e 1}
p(t,x)=\cF^{-1} \left[\exp \left( -\int_0^t a^{ij}(r) dr  \xi^i \xi^j \right) \right](x).
\end{align}
Note that  $p(t,x)$ is well defined since $\int^t_0 \delta(s)ds>0$ and 
$$
\left|\exp\left(-\int_0^t a^{ij}(r) dr  \xi^i \xi^j  \right) \right| \leq  \exp \left(-\int_0^t \delta(r) dr |\xi|^2 \right).
$$
It is easy to check that the representation formula (\ref{formula}) gives the unique solution $u\in \bL_p(T)$ to equation (\ref{eqn homo-2})  if $u_0\in \cS(\fR^d)$, and the general case also holds due to a standard approximation argument based on estimate (\ref{eqn classic}),
where $\cS(\fR^d)$ denotes the Schwartz space on $\fR^d$.

As in (\ref{del j}), define
$$
p_j(t,x):= \Delta_j p (t,x):=\Delta_j p (t,\cdot) (x):= \cF^{-1} \left[ \hat \Psi (2^{-j} \xi) \cF \left[p(t,\cdot) \right](\xi) \right] (x).
$$

 \begin{lemma}
                    \label{kernel l1 est}
Let $k \in \bZ$ and $\gamma \in \fR$. 
Suppose that  Assumptions \ref{co as}(ii) and \ref{ini as}(iii) hold. 
Then for all $t>0$
$$
\| \Delta^{\gamma/2}  p_k (t,\cdot) \|_{L_1}
\leq N2^{k\gamma} \exp\left(- c \int_0^t \delta(s)ds \cdot 2^{2k} \right),
$$
where $c$ and $N$ are some positive constants depending only on $d$, $\gamma$, and $\bar N_0$.
\end{lemma}
\begin{proof}
Denote
$$
q_k(t,x)=\cF^{-1} \left[|\xi|^{\gamma} \hat \Psi ( \xi ) \exp\left(-\int_{0}^{t} a^{ij}(s)ds \cdot 2^{2k}\xi^i \xi^j \right) \right](x),
$$
and 
$$\hat q_k (t,\xi) = \cF[q_k(t,\cdot)](\xi).
$$
Then considering the Fourier transform, one can easily check
\begin{align}
					\label{pq rel}
\Delta^{\gamma/2}p_k(t,x)= 2^{kd} 2^{k\gamma} q_k (t,2^k x).
\end{align}
Obviously, for all constants $a \geq 0$ and $c>0$,
\begin{equation}
   \label{ob}
\sup_{x>0} x^{a} e^{-cx}<\infty.
\end{equation}
Recalling that  $\hat \Psi (\xi )$ has the support in $B_2 \setminus B_{1/2}$,
differentiating $\hat{q}_k$ for $2d$-times, and using   (\ref{ob}), Assumptions \ref{co as}(ii) and \ref{ini as}(iii), one can find positive constants $N$ and $c$ such that for all $(t,\xi) \in (0,T) \times \fR^d$,
\begin{align}
                    \notag
\Big| \hat q_k(t,\xi)\Big|+\Big|\Delta^d_{\xi} \hat q_k(t,\xi)\Big|
&\leq N 1_{(2^{-1},2)}(|\xi|)\exp\left(- 4c \int_0^t \delta(s)ds \cdot |2^k \xi|^2\right).
\end{align}
Thus
\begin{align*}
\left|\left( 1+|x|^{2d}\right)q_j(t,x)\right|
&= \left|\cF^{-1} \left[ \left(1+\Delta^d_\xi \right) \hat q_k(t,\xi)\right](x) \right| \\
&\leq N\sup_{ |\xi| \in (2^{-1},2)}\left|  \exp \left(- 4c \int_0^t \delta(s)ds \cdot |2^k \xi|^2 \right) \right| \\
&\leq N\exp \left(- c \int_0^t \delta(s)ds \cdot 2^{2k}  \right),
\end{align*}
and
\begin{align*}
\|q_k(t,\cdot)\|_{L_1}&=\int_{|x|\leq 1} |q_k(t,x)|dx+ \int_{|x|\geq 1} \left(|x|^{2d} |q_k(t,x)| \right) |x|^{-2d} dx \\
& \leq N\exp \left(- c \int_0^t \delta(s)ds \cdot 2^{2k}  \right).
\end{align*}
Therefore by (\ref{pq rel}),
\begin{align*}
\|\Delta^{\gamma/2} p_k(t,\cdot)\|_{L_1} \leq N 2^{k\gamma} \exp \left(- c \int_0^t \delta(s)ds \cdot 2^{2k}  \right).
\end{align*}
The lemma is proved.
\end{proof}
Let $f \in \cS(\fR^d)$ and denote $f_i = \Delta_i f$.
Note that Littlewood-Paley operators have the following  orthogonal property:
\begin{align}
			                    \notag
&\Delta^{\gamma/(2p)}p(t,\cdot) \ast f(x) \\
						\notag
&=\left(\sum_{ j \in \bZ}^\infty \Delta^{\gamma/(2p)}p_j (t,\cdot)\right) \ast \left(S_0(f) + \sum_{i=1}^\infty f_i \right)(x) \\
                    \label{little ortho}
&=\sum^{-1}_{j=-\infty}  \Delta^{\gamma/(2p)}p_j(t,\cdot) \ast S_0(f)(x)
+\sum^\infty_{i =1} \sum_{ i-1 \leq j \leq i+1 } \Delta^{\gamma/(2p)}p_j(t,\cdot) \ast f_i(x).
\end{align}
This is because the intersection of the supports of  $\hat \Psi(2^{-i} \xi)$ and $\hat \Psi (2^{-j} \xi)$ are nonempty only if  $i-1 \leq j \leq i+1$.
In particular,
\begin{align}
                    \label{little ortho 2}
p(t,\cdot) \ast f(x) 
=p(t,\cdot) \ast \psi \ast S_0(f)(x)
+\sum^\infty_{i =1} \sum_{ i-1 \leq j \leq i+1 }p_j(t,\cdot) \ast f_i(x),
\end{align}
where
$$
\psi(x) :=  \cF^{-1} \left[  \sum_{j=-\infty}^{-1}  \hat \Psi (2^{-j} \xi) \right](x).
$$
\begin{lemma}
					\label{0504 lem 1}
Let $T \in (0,\infty)$ and $\gamma  \in \left[ 0 ,   \frac{2}{\beta p}\right]$.
Suppose that Assumptions \ref{co as} and \ref{ini as} hold.  Then for any  $f \in \cS(\fR^d)$,  
\begin{equation*}
\| \Delta^{\gamma/2} \cT_t f\|^p_{(0,T) \times \fR^d}\leq N \left( \|S_0(f)\|^p_{L_p}+ \sum_{j=1}^\infty \|f_j\|^p_{L_p} \right),
\end{equation*}
where $N$ is  a positive constant depending  only on $d$, $p$, $\beta$, $N_0$,$\bar N_0$, $T$, and $\int_{0}^{t_0} \delta(s)ds$. 
\end{lemma}
\begin{proof}
By \eqref{little ortho} and Young's convolution inequality,
\begin{align*}
\int_0^T \|\Delta^{\gamma/2} \cT_t f (t,\cdot)\|^p_{L_p} dt
&\leq \int_0^T \left( \sum^1_{j=-\infty} \|\Delta^{\gamma/2}p_j(t,\cdot)\|_{L_1}  \|S_0(f)\|_{L_p} \right)^p dt \\
\quad &\quad +\int_0^T \left(\sum^\infty_{i=1} \sum_{ i-1 \leq j \leq i+1 } \|\Delta^{\gamma/2}p_j(t,\cdot)\|_{L_1}  \|f_i\|_{L_p} \right)^p dt.
\end{align*}
By Lemma \ref{kernel l1 est}, 
\begin{align*}
\int_0^T \left( \sum^{-1}_{j=-\infty} \|\Delta^{\gamma/2}p_j(t,\cdot)\|_{L_1}  \|S_0(f)\|_{L_p} \right)^p dt 
&\leq \|S_0(f)\|^p_{L_p} \left(\sum^{-1}_{j = -\infty} 2^{j\gamma}\right)^p.
\end{align*}
Obviously, if $\gamma>0$ then
$$
 \left(\sum^{-1}_{j = -\infty} 2^{j\gamma}\right)^p \leq N.
$$
For the case $\gamma =0$, we apply \eqref{little ortho 2} and get
\begin{align*}
\int_0^T \| \cT_t f (t,\cdot)\|^p_{L_p} dt
&\leq \int_0^T \left(\| p(t,\cdot)\|_{L_1}  \|\psi \ast S_0(f)\|_{L_p} \right)^p dt \\
\quad &\quad +\int_0^T \left(\sum^\infty_{i=1} \sum_{ i-1 \leq j \leq i+1 } \|p_j(t,\cdot)\|_{L_1}  \|f_i\|_{L_p} \right)^p dt.
\end{align*}
Since $p(t,x)$ is the probability density function of the stochastic process $X_t$ introduced in (\ref{pdf x}), 
\begin{align}
					\label{e 1028 1}
\|p(t,\cdot)\|_{L_1}=1.
\end{align}
Moreover, since $\psi(x)$ is contained in the Schwartz class,
\begin{align}
					\label{e 1028 2}
\|\psi \ast S_0(f)\|_{L_p}  \leq N \| S_0(f)\|_{L_p}.
\end{align}
Combining (\ref{e 1028 1}) and (\ref{e 1028 2}), we have
\begin{align*}
\int_0^T \left(\| p(t,\cdot)\|_{L_1}  \|\psi \ast S_0(f)\|_{L_p} \right)^p dt 
\leq N\int_0^T  \|S_0(f)\|_{L_p}^p dt.
\end{align*}
Therefore, it is sufficient to estimate the following term
$$
\int_0^T \left(\sum^\infty_{i=1} \sum_{ i-1 \leq j \leq i+1 } \|\Delta^{\gamma/2}p_j(t,\cdot)\|_{L_1}  \|f_i\|_{L_p} \right)^p dt.
$$
By Lemma \ref{kernel l1 est}, the above term is less than or equal to constant times of 
\begin{align*}
\cI:=\int_0^T \left(\sum^\infty_{j=1} 2^{j\gamma} e^{-(c/4)\int_0^t \delta(s)ds~ \cdot 2^{2j}} \|f_j\|_{L_p} \right)^p dt.
\end{align*}
To estimate $\cI$, we consider the decomposition 
$$
\cI=\int^{T}_{t_0} \cdots dt + \int^{t_0}_0 \cdots dt.
$$
Define $\kappa_0:=\int^{t_0}_0 \delta(s) ds>0$ and observe 
$$
\sum_{j\geq 1} 2^{j\gamma q} e^{-cq\kappa_0 2^{2j}} <\infty,
$$
where $q:=p/(p-1)$. Therefore, by H\"older's inequality
\begin{align*}
& \int_{t_0}^T \left(\sum^\infty_{j=1} 2^{j \gamma} e^{-c\int_0^t \delta(s)ds~ \cdot 2^{2j}} \|f_j\|_{L_p}\right)^pdt \leq T 
 \left(\sum^\infty_{j=1} 2^{j \gamma} e^{-c\kappa_0 \cdot 2^{2j}} \|f_j\|_{L_p}\right)^p \\
&\leq T  \left(\sum^\infty_{j=1} 2^{j\gamma q} e^{-cq 2^{2j}}\right)^{1/q} \sum_{j=1}^{\infty} \|f_j\|^p_{L_p} \leq N\sum_{j=1}^{\infty} \|f_j\|^p_{L_p} .
\end{align*}
Thus it only remains to show 
\begin{align*}
\int_0^{t_0} \left(\sum^\infty_{i=1} 2^{j\gamma} e^{-(c/4)\int_0^t \delta(s)ds~ \cdot 2^{2j}} \|f_j\|_{L_p}\right)^p dt
\leq N\sum^\infty_{i=1} \|f_j\|^p_{L_p} dt.
\end{align*}
Note that
\begin{align*}
&\int_0^{t_0} \left(\sum^\infty_{j=1} 2^{j\gamma} e^{-(c/4)\int_0^t \delta(s)ds~ \cdot 2^{2j}} \|f_j\|_{L_p}\right)^pdt \leq 2^p (\cI_{1} + \cI_{2}),
\end{align*}
where
$$
\cI_{1}=\int_0^{t_0} \left(\sum_{2^{2j}\int_0^t \delta(s)ds \leq 1} 2^{j\gamma} e^{-c\int_0^t \delta(s)ds~ \cdot 2^{2j}} \|f_j\|_{L_p}\right)^pdt
$$
and
$$
\cI_{2}=\int_0^{t_0} \left(\sum_{2^{2j}\int_0^t \delta(s)ds> 1} 2^{j\gamma} e^{-c\int_0^t \delta(s)ds~ \cdot 2^{2j}} \|f_j\|_{L_p}\right)^pdt.
$$
First we estimate $\cI_{1}$.
Observe that  $2^{2j} \int_0^t \delta(s)ds   \leq 1$ if and only if $j \leq  \log_2 \left(\int_0^t \delta(s)ds \right)^{-1/2}$.  
Using H\"older's inequality and Fubini's theorem (also use $1 =2^{\frac{j}{\beta p}} 2^{-\frac{j}{\beta p} }$), we get (recall $q=p/(p-1)$),
\begin{align*}
\cI_{1}
&\leq \int_0^{t_0} \left(\sum_{\int_0^t \delta(s)ds~ \cdot 2^{2j} \leq 1} 2^{ \frac{q j}{\beta p}} \right)^{p-1} 
\sum^\infty_{\int_0^t \delta(s)ds~ \cdot 2^{2j} \leq 1} 2^{ \left(\gamma p- \frac{1}{\beta} \right) j}  \|f_j\|^p_{L_p} dt \\
&\leq N \int_0^{t_0} \left(\int_0^t \delta(s)ds \right)^{-\frac{1}{2\beta}} 
\sum_{\int_0^t \delta(s)ds~ \cdot 2^{2j} \leq 1} 2^{ \left(\gamma p- \frac{1}{\beta} \right) j}  \|f_j\|^p_{L_p} dt \\
&\leq N \sum^\infty_{j=1} 2^{ \left(\gamma p- \frac{1}{\beta} \right) j}  \|f_j\|^p_{L_p}  \int_0^{t_0}
1_{ \int_0^t \delta(s)ds \cdot 2^{2j} \leq 1}(t) \left(\int_0^t \delta(s)ds \right)^{-\frac{1}{2\beta}} dt.
\end{align*}
Due to Assumption \ref{ini as}(ii),
\begin{align*}
&\int_0^{t_0}1_{ \int_0^t \delta(s)ds \cdot 2^{2j} \leq 1}(t) \left(\int_0^t \delta(s)ds \right)^{- \frac{1}{2 \beta}} dt \\
&\leq \sum_{m=0}^\infty \left| \left\{ t \in (0,t_0) : 2^{-2(j+m+1)}\leq \int_0^t \delta(s)ds \leq 2^{-2(j+m)} \right\} \right|
2^{ \frac{j+m+1}{\beta}} \\
&\leq N 2^{ -\frac{j}{\beta}  } \sum_{m=0}^\infty 2^{-\frac{m}{\beta} }\leq N 2^{ -\frac{j}{\beta}}.
\end{align*}
Therefore,
$$
\cI_{1} 
\leq N \sum^\infty_{j=1} 2^{ \left(\gamma p- \frac{2}{\beta} \right) j} \|f_j\|^p_{L_p}
\leq  N \sum^\infty_{j=1} \|f_j\|^p_{L_p}.
$$
Next we estimate $\cI_2$.
For $b \geq 0$, using the equality $1= 2^{bj} 2^{-bj}$ and the H\"older inequality, we get 
\begin{align*}
&\cI_{2} 
\leq \int_0^{t_0} \left(\sum_{\int_0^t \delta(s)ds~ \cdot 2^{2j} > 1} 2^{ ( \gamma +b) q j}  e^{-cq\int_0^t \delta(s)ds~ \cdot 2^{2j}} \right)^{p-1} \sum^\infty_{\int_0^t \delta(s)ds~ \cdot 2^{2j} > 1} 2^{-    b p j }  \|f_j\|^p_{L_p} dt.
\end{align*}
By (\ref{ob}), for any positive constant $a > \frac{\gamma + b}{2}$,
$$
e^{-cq \int^t_0 \delta(s)ds \cdot 2^{2j}}\leq N 2^{-2  a q j } \left(\int^t_0 \delta(s)ds\right)^{-a q}.
$$
Thus by Assumption \ref{ini as}(ii),
\begin{align}
					\notag
&\cI_{2}\leq \int_0^{t_0} \left( \left( \int_0^t \delta(s) ds \right)^{-a q}
 \sum_{\int_0^t \delta(s)ds~ \cdot 2^{2j} > 1} 2^{ ( \gamma +b -2a) q j}  \right)^{p-1}  \\
					\notag
&\qquad \times  \sum^\infty_{\int_0^t \delta(s)ds~ \cdot 2^{2j} > 1} 2^{-b p j  }  \|f_j\|^p_{L_p} dt \\
					\notag
&\leq N\int_0^{t_0} \left( \int_0^t \delta(s) ds \right)^{-\frac{(\gamma + b)p}{2}} 
\sum^\infty_{\int_0^t \delta(s)ds~ \cdot 2^{2j} > 1} 2^{-bpj}  \|f_j\|^p_{L_p} dt \\
					\notag
&\leq N \sum^\infty_{j=1}  2^{- b p j}  \|f_j\|^p_{L_p}  
\int_{0}^{t_0} \left( \int_0^t \delta(s) ds \right)^{-\frac{(\gamma + b)p}{2}}
1_{\int_0^t \delta(s)ds~ \cdot 2^{2j} > 1}(t) dt\\
					\notag
&\leq N \sum^\infty_{j=1}  \sum_{m=0}^\infty   2^{ \gamma p j} 
2^{- (\gamma +b) p m } \|f_j\|^p_{L_p}  \\
					\notag
&\qquad \times \left| \left\{ t \in [0,t_0] : 2^{-2(j-m)} \leq  \int_0^t \delta(s)ds\leq 2^{-2(j-m-1)} \right\} \right| \\
					\label{e 1029 1}
&\leq N\sum_{m=0}^{\infty} 2^{ \left( \frac{2}{\beta} - (\gamma+b) p \right) m}  
\sum^\infty_{j=1}  2^{ \left(  \gamma p - \frac{2}{\beta}   \right) j } \|f_j\|^p_{L_p}
\end{align}
Take $b = \frac{4}{\beta p}$. Then
$$
\frac{2}{\beta} - (\gamma+b) p < 0.
$$
Therefore, from (\ref{e 1029 1}), we have
$$
\cI_{2} \leq N\sum^\infty_{j=1}  2^{ \left(  \gamma p - \frac{2}{\beta}   \right) j } \|f_j\|^p_{L_p} 
\leq N \sum^\infty_{j=1}  \|f_j\|^p_{L_p}.
$$
The lemma is proved.
\end{proof}

 We continue the proof of the theorem and  assume $u_0 \in \cS(\fR^d)$. 
Define $v=(1-\Delta)^{(1-1/(\beta p))}u$
and $v_0:=  (1-\Delta)^{(1-1/(\beta p))} u_0$.
Then obviously $v_0 \in \cS(\fR^d)$ and $v$ satisfies 
$$
v_t=a^{ij}(t)v_{x^ix^j}, \quad t \in (0,T); \quad v(0,\cdot)=v_0.
$$
Thus by (\ref{formula}) and  Lemma \ref{0504 lem 1},
 \begin{align}
						\label{e 1028 4}
\|v\|_{\bH_p^{2/(\beta p)}}  
\leq N \left( \|v\|_{\bL_p} +\| \Delta^{1/(\beta p)} v\|_{\bL_p}    \right)
&\leq N \left( \|S_0 (v_0)\|_{L_p}+ \sum_{j=1}^\infty \|   (v_0)_j\|^p_{L_p} \right),
 \end{align}
where the first  inequality can be easily induced from the classical multiplier theorem (e.g. \cite[Theorem 5.2.7]{grafakos2008classical}).
Since  the operator $(1-\Delta)^{(1-1/(\beta p))}$ is  an isometry  from $H^{2/(\beta p)}$ (resp. $B^0_p$) to $H^2_p$ 
(resp. $B_p^{2 \left(1-1/(\beta p) \right)}$) (see \cite[Theorem 6.2.7]{bergh1976interpolation} ), 
it follows from (\ref{e 1028 4}) that 
  \begin{equation}
     \label{real final}
  \|u\|_{\bH^2_p(T)}\leq N \|u_0\|_{B^{2 \left(1-1/(\beta p) \right)}_p}.
  \end{equation}
  For general $u_0$, it is enough to apply a standard approximation based on (\ref{real final}). The theorem is proved.  \qed

%
%

\mysection{Acknowledgement}
We are sincerely grateful to the referee.
Especially, we could considerably enhance Theorem \ref{main thm 1} due to the referee.

\end{document}